\documentclass{CUP-JNL-NMJ}%
\usepackage{graphicx}
\usepackage{lmodern}
\usepackage{multicol,multirow}
\usepackage{amsmath,amssymb,amsfonts}
\usepackage{mathrsfs}
\usepackage{amsthm}
\usepackage{tikz-cd}
\usepackage{rotating}
\usepackage{appendix}
\usepackage[numbers]{natbib}
\usepackage{ifpdf}
\usepackage{xcolor}
\usepackage[colorlinks,allcolors=blue]{hyperref}
\usepackage{lipsum}

\theoremstyle{cupproof}
\newtheorem*{proof}{Proof}

\numberwithin{equation}{section}

\articletype{RESEARCH ARTICLE}
\jrarticle{Nagoya Math. J.}
\artid{XX}
\jyear{20XX}
\jvol{XX}
\jdoi{}
\jrcopyrightline{The Author(s), 2023. Published by Cambridge University Press on behalf of Foundation Nagoya Mathematical Journal.}

\newcommand{\N}{\mathbb N}
\newcommand{\Z}{\mathbb Z}

\DeclareMathOperator{\Pic}{Pic}
\DeclareMathOperator{\Hom}{Hom}

\DeclareMathOperator{\Spec}{Spec}

\newcommand{\hop}{\vskip .3cm\noindent} 

\theoremstyle{definition}
\newtheorem{thm}{Theorem}[section] 
\newtheorem{cor}[thm]{Corollary}
\newtheorem{prop}[thm]{Proposition}
\newtheorem{lem}[thm]{Lemma}
\newtheorem{defi}[thm]{Definition}
\newtheorem{rema}[thm]{Remark}
\newtheorem{exam}[thm]{Example}
\newtheorem{Counter-example}[thm]{Counter-example}
\makeatletter
\newcommand{\setword}[2]{%
  \phantomsection
  #1\def\@currentlabel{\unexpanded{#1}}\label{#2}%
}
\begin{document}

\begin{Frontmatter}

\title[Extending torsors under quasi-finite flat group schemes]{Extending torsors under quasi-finite flat group schemes}

\author{Sara Mehidi}

\authormark{S. Mehidi}



\amsclassification[2020 Mathematics subject classification]{14A15, 14A21, 14B25, 14L15.}

\abstract{Let $R$ be a discrete valuation ring of field of fractions $K$ and of residue field $k$
of characteristic $p > 0$.\\
In an earlier work, we studied the question of extending torsors over $K$-curves into torsors over
$R$-regular models of the curves in the case when the structural $K$-group scheme of the torsor admits a finite flat model over $R$. In this paper, we first give a simpler description of the problem in the case where the curve is semistable using recent work in \cite{Holmes} and \cite{MW}. Secondly, if $R$ is assumed to be Henselian and Japanese, we solve the problem of extending torsors by combining our previous work together with results in \cite{Antei} and \cite{Pedro}, including the case where the structural group does not admit a finite flat $R$-model.}

\end{Frontmatter}

\section{Introduction}

All over this paper, $R$ denotes a discrete valuation ring with field of fractions $K$ and residue field $k$ of characteristic $p >0$. In addition, schemes and log schemes are supposed to be locally noetherian.\\

Let $S$ be a regular scheme and $U \subseteq S$ a dense open subset. Let $f: X \to S$ be a finite flat morphism of schemes, unramified over $U$. The Zariski-Nagata purity theorem, known as \textit{purity of the branch locus}, says that the closed subset of $S$ where $f$ ramifies is either empty or of pure codimension $1$. On the other hand, given a finite étale group scheme $G/S$, and an fppf $G_U$-torsor $X \to U$ (hence an étale torsor by fppf descent), if it extends into an fppf $G$-torsor over the whole $S$, it needs to be étale, hence unramified. But the purity theorem suggests that such an extension may not exist in general. Nevertheless, if the extension of the torsor $X \to U$ ramifies outside $U$, and if the ramification is tame, there might be a way to lift it into a \textit{log} torsor over $S$. Indeed, assume that $D:=S\backslash U$ is a normal crossing divisor, so that one can endow $S$ with the divisorial log structure induced by $D$. Then, \textit{logarithmic torsors} over $X$ are, roughly speaking, tamely ramified over $D$. This approach of extending torsors into log torsors has been followed in \cite{Sara}, and the main purpose of this paper is to enhance their results. The paper is divided into two independent parts, which we explain below. \\

\textbf{Part I:}

Let $C$ be a smooth projective curve over $K$, endowed with a $K$-point $Q$ and let $J$ denote its Jacobian variety. Let $\mathcal{C}$ be a regular model of $C$ over $R$, such that its special fiber is a normal crossing divisor, and endow $\mathcal{C}$ with the canonical log structure induced by this divisor; let $\mathcal{Q}$ denote the $R$-section extending $Q$ by properness. Let $G$ be a finite commutative group scheme over $K$. It is well-known that the Jacobian variety classifies fppf commutative torsors, which can be rephrased through the one-to-one correspondence (cf. \cite[Lemme 2.3]{Sara}): 
\begin{equation*}
 H_{fppf}^1(C,Q,G) \simeq \Hom(G^D,J) \label{Jac}  \tag{$\star$} 
\end{equation*}

where the group on the left is the first cohomology group classifying fppf pointed $G$-torsors (relatively to $Q$) over $C$, and $G^D$ is the Cartier dual of $G$. It is shown in \cite[Remark 1.11]{Sara} that one has a similar correspondence for log torsors over $\mathcal{C}$. Indeed, if $\mathcal{G}$ denotes a finite flat commutative $R$-group scheme, one has a one-to-one correspondence:
$$ H_{klf}^1(\mathcal{C},\mathcal{Q},\mathcal{G}) \simeq \mathrm{Hom}(\mathcal{G}^D,\mathrm{Pic}^{log}_{\mathcal{C}/R})  $$

where the group on the left is the first cohomology group classifying Kummer log flat pointed  $\mathcal{G}$-torsors (relatively to $\mathcal{Q}$) over $\mathcal{C}$, and $\mathrm{Pic}^{log}_{\mathcal{C}/R}$ is the relative log Picard functor of $\mathcal{C}/R$. An immediate consequence of this is that, given a pointed fppf $G$-torsor over $C$, it extends into a log torsor over $\mathcal{C}$ if and only if there exists a finite flat $R$-model $\mathcal{G}$ of $G$ such that the $K$-morphism $G^D \to J$ from (\ref{Jac}) corresponding to the torsor extends into an $R$-morphism $\mathcal{G}^D \to \mathrm{Pic}^{log}_{\mathcal{C}/R}$. Moreover, if $\mathcal{J}$ is the Néron model of $J$ over $R$, it is shown in the same paper that the canonical map $J \hookrightarrow \mathrm{Pic}^{}_{{C}/R}$ extends uniquely into a map $\mathcal{J} \to \mathrm{Pic}^{log}_{\mathcal{C}/R}$. In particular, if the morphism $G^D \to J$ extends into a morphism $\mathcal{G}^D \to \mathcal{J}$, the torsor extends into a $\mathcal{G}$-log torsor over $\mathcal{C}$. In this paper, we want to invest the converse. Given that $\mathcal{J}$ is a smooth scheme, it is a \textit{nicer} object to work with than the log Picard functor. 
 Using the results of \cite{MW} on log curves, we give a partial answer to this question:\\

\textbf{Corollary 3.4.~~}\textsl{~~Let $C$ be a smooth projective semistable and geometrically connected curve endowed with a $K$-point. Let $\mathcal{C}$ be an $R$-regular model of $C$ with normal crossing special fiber and endowed with the divisorial log structure (cf. example \ref{log sch}(3)). Let $G$ be a finite commutative $K$-group scheme and $\mathcal{G}$ a finite flat $R$-model of $G$. Then a pointed fppf $G$-torsor over $C$ extends into a pointed $\mathcal{G}$-log torsor over $\mathcal{C}$ if and only if the $K$-morphism $G^D \to J$ associated to the generic torsor (cf. (\ref{Jac})) extends into an  $R$-morphism $\mathcal{G}^D \to \mathcal{J}$.}\\


\textbf{Part II:}

In the second part of this paper, we would like to drop the assumption that $G$ admits a finite flat $R$-model. Indeed, there exist groups which do not admit such a model (cf. see Example \ref{CE}). However, if $G$ is finite, then what is true in general is that it admits a quasi-finite flat $R$-model (cf. \cite[Theorem 3.7]{Antei}). In the latter, the authors took advantage of this to obtain partial answers to the problem of extending finite torsors. In particular, they showed that there exists a modification (a Néron blow-up) of the regular model of the curve over which the torsor extends under some quasi-finite flat group scheme. On the other hand, under additional assumptions on $R$, it is shown in \cite{Pedro} that a torsor under a quasi-finite flat group scheme reduces into a torsor under a finite flat group scheme. Combining these two results, together with our previous work, we prove the following: \\

\textbf{Theorem 4.4.~~}\textsl{Let $C$ be a smooth projective and geometrically connected $K$-curve with a $K$-point $Q$. Let $\mathcal{C}$ be an $R$-regular model of $C$, and $G$ a finite commutative $K$-group scheme. Let $Y \to C$ be an fppf pointed $G$-torsor (relatively to $Q$). Then, there exists a quasi-finite flat group scheme $\mathcal{G}$ over $R$ with generic fiber $G$, and an fppf pointed $\mathcal{G}$-torsor $\mathcal{Y} \to \mathcal{C}$ that extends the $G$-torsor $Y \to C$.}\\


\section{Preliminaries}\label{prelim}
\subsection{Kummer log flat torsors}\label{tor log def}

\subsubsection{Log schemes.}  Let $({X},\mathcal{O}_X)$ denote a scheme. A logarithmic (log) structure on $X$ is the data of a sheaf of monoids  $M_X$ on $X_{\textrm{\'et}}$, together with a morphism $\alpha_X: M_X \to \mathcal{O}_X$ such that $\alpha_X^{-1}(\mathcal{O}_X^{\times}) \simeq \mathcal{O}_X^{\times}$. A scheme endowed with a log structure is said to be a logarithmic (log) scheme.\\
A morphism of log schemes is a morphism $f : X \to Y$ of the underlying schemes, together with a morphism $f^{-1}M_Y \to M_X$ such that the diagram

\[
\begin{tikzcd}
f^{-1}M_Y \arrow{r}{} \arrow[swap]{d}{f^{-1}\alpha_Y} & M_X \arrow{d}{\alpha_X} \\
f^{-1}\mathcal{O}_Y  \arrow{r}{} & \mathcal{O}_X
\end{tikzcd}
\]
commutes.

\subsubsection{Charts.}
If $\underline{P}$ is the constant sheaf associated to a monoid $P$, and if we are given a morphism of sheaves $\underline{P} \to \mathcal{O}_X$, it induces a unique log structure on $X$ \cite[Proposition 1.1.5]{Ogus}. If $(X,M_X)$ is a log scheme, it is said to have a \textit{chart} on $P$ if the log structure induced by $P$ is isomorphic to $M_X$. All the log schemes in this paper are supposed to admit charts étale locally. Furthermore, if $P$ is fine (finitely generated and integral, i.e. $P \hookrightarrow P^{gp}$) and saturated (i.e. if $a \in P^{gp}$ such that $a^n \in P$ for some non-zero integer $n$, then $a \in P$), $X$ is said to be a \textit{fine and saturated} log scheme; we refer to \cite{Ogus} for further details.\\

\subsubsection{Inverse image log structure and strict morphisms.}

If $Y$ is a log scheme with underlying scheme $\underline{Y}$, and $f: X \to \underline{Y}$ is a morphism of schemes, then the composition $f^{-1}M_Y \xrightarrow{f^{-1}\alpha_Y}  f^{-1}\mathcal{O}_Y \to \mathcal{O}_X$ is a prelog structure on $X$, and induces \textit{the inverse image log structure} on $X$ that we denote by $f^{*}M_Y$.\\
If $f: X \to Y$ is a morphism of log schemes, the map $f^{-1}M_Y \to M_X$ factors canonically through $f^{*}M_Y \to M_X$.\\
The morphism of log schemes $f:X \to Y$ is said to be strict if the induced map $f^*M_Y \to M_X$ is an isomorphism.\\

\subsubsection{Direct image log structure.}
If $f : X \to Y$ is a morphism of schemes and $\alpha_X : M_X \to \mathcal{O}_X$ is a log structure on $X$, then the natural map $\beta$ in the diagram below
\[
\begin{tikzcd}
f_* M_X \times_{f_* \mathcal{O}_X}\mathcal{O}_Y \arrow{r}{\beta} \arrow[swap]{d}{} & \mathcal{O}_Y \arrow{d}{} \\
f_* M_X  \arrow{r}{f_*\alpha_X} & f_* \mathcal{O}_X
\end{tikzcd}
\]
is a log structure on $Y$, called \textit{the direct image log structure} induced by $\alpha_X$. We denote it by $f^{log}_* \alpha_X: f^{log}_* M_X \to \mathcal{O}_Y$.

\begin{exam}\label{log sch}\hop
\begin{enumerate}
\item Let $X$ be a scheme. $M_X:=\mathcal{O}_X^{\times}$ defines a fine and saturated log structure on $X$ called the trivial log structure. $X$ has a chart on the monoid $\{1\}$. If $X$ is a log scheme, the largest Zariski open subset of $X$ (possibly empty) on which the
log structure is trivial is called the \textit{open of triviality} of $X$.

\item Let $X$ be a regular scheme and let $j: U \hookrightarrow X$ be a dense open subset whose complementary is a normal crossing divisor $D$ on $X$. Then the sheaf
$$M_X(V):= \{s \in \mathcal{O}_X(V) | s_{|V \cap U} \in \mathcal{O}_{V \cap U}(V \cap U)^{\times}\} \hookrightarrow \mathcal{O}_X(V)$$
defines the \textit{divisorial log structure} on $X$. \\
It is the same as the direct image log structure on $X$ of the trivial log structure on $U$, i.e. $j^{log}_*\mathcal{O}_U^{\times}$ (cf. \cite[\S 1.5]{Kato}).
It is a fine and saturated log structure on $X$.\\

 Note that $U$ is the open of triviality of the divisorial log structure on $X$. In particular, $\Spec(R)$ can be seen as a fine and saturated log scheme with the log structure induced by $\Spec(k)$ seen as a divisor. $\Spec(R)$ has a chart on $\N$ given by $\N \to R; 1 \mapsto \pi$ , where $\pi$ is the uniformizer of $R$. More generally, if ${X}$ is a flat $R$-scheme such that its special fiber is a normal crossing divisor, then it can be seen as a fine and saturated log scheme with the log structure induced by its special fiber. The generic fiber ${X}_K$ is the open of triviality of the log structure. Furthermore, it has (étale) locally a chart on $\mathbb{N}^r$. \\

\item If $S$ is a fine and saturated log scheme, we say that $ {X} \to S$ is a log curve if it is a proper, integral (cf. \cite[Definition 2.3]{FFKato}), vertical\footnote{vertical means that the curve doesn't have marked points (cf. \cite[\S 1.8 (2)]{FKato}.)}, log smooth morphism of (fine and saturated) log schemes with connected and reduced geometric fibres of pure dimension $1$. Then, according to \cite{FKato}, the underlying
scheme of ${X}$ is a flat family of nodal curves over $S$ and one has an explicit description of the log structure on the geometric points of ${X}$ lying above geometric points of $S$ . This is a fine and saturated log structure on $X$. 
Conversely, if $S$ is the spectrum of a DVR $R$ and $K=\mathrm{Frac}(R)$, if $C/K$ is a semistable curve with $\mathcal{C}$ some regular model over $S$, there exist canonical log structures on $\mathcal{C}$ and $S$ making $\mathcal{C} \to S$ a log curve (cf. \cite[\S 3]{Olsson}). Furthermore, if $S$ is endowed with its divisorial log structure, and if the special fiber of $\mathcal{C}/S$ is a normal crossing divisor, tho so-mentioned canonical log structure on $\mathcal{C}$ agrees with the divisorial one (cf. \cite[\S 2 of proof of Lemma 2.2.5.1 and Theorem 2.4.1.3]{MW}).

\end{enumerate}

\end{exam}

\hop \hop

We consider here the category of fine and saturated log schemes, endowed with the \textit{Kummer log flat} topology (we sometimes write \textit{klf} to refer to this topology for simplicity). We refer to \cite{Kato} or \cite[\S 2.2]{Gill2} for the definition of this Grothendieck topology. A torsor in this category, defined with respect to the klf topology, is called a \textit{logarithmic torsor} (or a log torsor). The structural group of the torsor is always assumed to be endowed with the strict log structure (the inverse log structure of that of the base). Moreover, a Kummer log flat cover of a scheme endowed with the trivial log structure is just a cover for the fppf topology. So, in this paper, the category of schemes is endowed with the fppf topology.\\

\begin{exam}\cite[\S 1 ; 1.9.3]{Katoo}\hop
Let $A$ be a discrete valuation ring with uniformizer $\pi$. Assume that it contains a primitive $n$-th root of unity and that $n \in \N$ is invertible in $A$. We set $B:=A[\sqrt[n]{\pi}]$, $X:=\Spec(A)$ and $Y:=\Spec(B)$. We endow both these schemes with the divisorial log structure, making them into fine saturated log schemes. Let $G:=\mathrm{Aut}_A(B)=\mu_n\simeq \Z/n\Z$. Then $Y \to X$ is not an fppf $G$-torsor (because it is totally ramified while $G$ is unramified) but it is a klf (more precisely, Kummer log étale) torsor.
\end{exam}

\subsection{Extension of torsors under a finite flat group scheme}\label{previous work}
We recall briefly in this section the main results of \cite{Sara} on the problem of extending torsors under finite flat group schemes.  From now on, $\Spec(R)$ is endowed with the divisorial log structure. Let $(fs/R)_{klf}$ denote the category of fine and saturated log schemes over $R$ endowed with the Kummer log flat topology, and let $(Sch/R)_{fppf}$ be the category of schemes over $R$ endowed with the fppf topology. The latter can be viewed as a full subcategory of $(fs/R)_{klf}$ by endowing an $R$-scheme with the inverse log structure of that of $\Spec(R)$. We recall the following definitions:

\begin{defi}\label{Piclog}
\begin{enumerate}
\item  We define the following functor
\begin{align*}
   \mathbb{G}_{m,log,R}: (fs/R)_{klf} &\to (Ab)\\
    T  & \mapsto \Gamma(T,M_T^{gp}).
\end{align*}
which is a sheaf in the klf site \cite[Theorem 3.2]{Kato}. Note that generically, it is isomorphic to $\mathbb{G}_{m,K}$.

\item Let $C$ be a smooth projective $K$-curve with an $R$-regular model $\mathcal{C}$ endowed with the divisorial log structure. Using the embedding $(Sch/R)_{fppf} \hookrightarrow (fs/R)_{klf}$, consider the following functor
\begin{align*}
    (Sch/R)_{fppf} &\to (Sets)\\
    T  & \mapsto \{\mathbb{G}_{m,log,\mathcal{C}}-log~torsors~on~\mathcal{C}_{T} \}.
\end{align*}

The \textit{log Picard functor}, denoted by $\Pic^{log}_{\mathcal{C}/R}$, is defined to be the fppf sheaffification on $(Sch/R)_{fppf}$ of the previous functor.
Furthermore, it is clear that its generic fiber is $\Pic_{C/K}$, the usual relative Picard functor of $C/K$.
\end{enumerate}

\end{defi}

\begin{thm}\label{previ} \cite[Remark 1.11]{Sara} Let $C$ be a smooth projective and geometrically connected curve over $K$, endowed with a $K$-point $Q$, and let $J$ denote its Jacobian variety. Let $\mathcal{C}$ be an $R$-regular model of $C$ such that its special fiber is a normal crossing divisor. Endow $\mathcal{C}$ with the divisorial log structure and let $\mathcal{Q}$ be the $R$-section that extends $Q$ over $\mathcal{C}$. Let $G$ be a finite commutative $K$-group scheme with finite flat $R$-model $\mathcal{G}$ and let $\mathcal{G}^D$ denote its Cartier dual. We have a canonical isomorphism:
 $$H^1_{klf}(\mathcal{C},\mathcal{Q},\mathcal{G})\xrightarrow{\simeq} \Hom(\mathcal{G}^D, \Pic^{log}_{\mathcal{C}/R})$$
 where $H^1_{klf}(\mathcal{C},\mathcal{Q},\mathcal{G})$ denotes the cohomology group that classifies logarithmic $\mathcal{G}$-torsors over $\mathcal{C}$, pointed relatively to $\mathcal{Q}$. In particular, a pointed fppf $G$-torsor (relatively to $Q$) over $C$ extends into a pointed log $\mathcal{G}$-torsor (relatively to $\mathcal{Q}$) over $\mathcal{C}$ if and only if the associated $K$-morphism $G^D \to J$ (cf. (\ref{Jac})) extends into an $R$-morphism $\mathcal{G}^D \to \Pic^{log}_{\mathcal{C}/R}$.
\end{thm}

\begin{prop}\label{PropN} \cite[Proposition 1.12 and Proposition 1.18]{Sara} With the assumptions of the previous theorem, if $\mathcal{J}$ is the Néron model of $J$ over $R$, then the closed immersion $J \hookrightarrow \Pic_{C/K}$ extends uniquely into an $R$-morphism $\mathcal{J} \to \Pic^{log}_{\mathcal{C}/R}$. In particular, if the associated $K$-morphism $G^D \to J$ of the generic torsor extends into an $R$-morphism $\mathcal{G}^D \to \mathcal{J}$, the torsor extends into a $\mathcal{G}$-log torsor over $\mathcal{C}$. Moreover, if $\mathcal{J}^0$ denotes the identity component of $\mathcal{J}$, the extended log torsor is fppf if and only if $\mathcal{G}^D \to \mathcal{J}$ factors through $\mathcal{J}^0$. 
\end{prop}



\section{Part I: Case of semistable curves}
\subsection{The Log Picard functor}\label{Picard}

Recently, the Picard log functor has been defined in a more general frame. Let $S$ be a log regular scheme and let $U \subseteq S$ be the open of triviality of the log structure on $S$ (which is non empty and even dense in $S$ by log regularity). Let ${X} \to S$ be a logarithmic curve (hence smooth over $U$). In \cite{MW}, following the ideas of Illusie and Kato, the authors constructed the analogue of the Picard functor in the logarithmic setting: the logarithmic Picard group that they denoted by $\mathrm{LogPic}_{{X}/S}$. It is the sheaf of isomorphism classes of the stack which parameterizes the logarithmic line bundles, i.e torsors under the group scheme $\mathbb{G}_{m,log,S}$ which verify a certain condition called \textit{the condition of bounded monodromy}.\\
Naturally, the logarithmic Picard group coincides with the ordinary Picard group over ${X}_U$, where the log structure is trivial. Furthermore, logarithmic line bundles have a natural notion of (total) degree extending the notion of degree of classical line bundles (cf. \cite[\S 4.5]{MW}).

Using this notion of degree, it is defined in \cite[Definition 3.47]{Holmes} $\mathrm{LogPic}^0_{{X}/S}$, the subsheaf of $\mathrm{LogPic_{{X}/S}}$ consisting of log line bundles of total degree zero, which they called \textit{the logarithmic Jacobian}.  In fact, this provides the best possible extension of the Jacobian $\Pic^0_{{X}_U/U}$. \\

Futhermore, one can restrict the functor $\mathrm{LogPic}^0_{{X}/S}$ to the category of schemes via the embedding $(Sch)_{fppf} \hookrightarrow (fs/S)_{klf}$, and the resulting functor is called the \textit{strict logarithmic Jacobian} and denoted by $\mathrm{sLogPic^0}_{{X}/S}$ (cf. \cite[Definition 4.5]{Holmes}).

\begin{thm}\cite[Corollary 6.13]{Holmes}\label{ner}
    $\mathrm{sLogPic}^0_{{X}/S}$ is the Néron model of $~\Pic^0_{{X}_U/U}$.
\end{thm}


\begin{rema}
The condition of bounded monodromy is essential to get a log Picard group of $X/S$ that is well-behaved in families. For the purposes of this paper, we don't need to recall its definition in the general setting; we will simply recall it in the case where the base $S$ is the spectrum of a discrete valuation ring endowed with the divisorial log structure. We will see that in this case, this condition is automatically satisfied.

So we take $S$ to be $\Spec(R)$ endowed with the divisorial log structure. Let $X$ denote an $R$-log curve and let $s=\Spec(\Bar{k})$ denote the geometric closed point of $\Spec(R)$. If $\Gamma$ denotes the dual graph (which is assumed to be oriented) of $X_{s}$ and if $M_R$ denotes the divisorial log structure over the base $\Spec(R)$, one can define on $\Gamma$ a \textit{length map} $l: \Gamma \to \overline{M_R}_{,s}$, where $\overline{M_R}:=M_R/\mathcal{O}_R^*$. Since $M_R=\mathcal{O}_R \cap \mathcal{O}_K^*$, one finds that $\overline{M_R}_{,s}\simeq \mathbb{N}$.\\
  The data $\mathcal{X}=(\Gamma, l: \Gamma \to \N)$ is called the tropical curve associated to $X_{s}$ (cf. \cite[\S 2.3]{MW}). In addition, on can define a topology on tropical curves (cf. \cite[\S 3]{MW}) which allows to do homology on them. In particular, the length map $l$ can be extended to $H_1(\mathcal{X})$.\\
  To any log line bundle over $X$ is associated a class of morphisms $H_1(\mathcal{X}) \to \overline{M_R}_{,s}^{gp}=\N^{gp}=\Z$, called the \textit{monodromy class} (cf. \cite[s 3.5 and \S 4.1]{MW}). A logarithmic line bundle is said to have bounded monodromy if for any $\gamma \in H_1(\mathcal{X}), \exists n \in \N$ such that $-nl(\gamma) \leq \alpha(\gamma) \leq nl(\gamma)$, where $\alpha: H_1(\mathcal{X}) \to \Z$ is some representative in the monodromy class of the line bundle (this condition does not depend on the choice of a representative). $\overline{M_R}_{,s}\simeq\mathbb{N}$ being archimedean, it is clear that the monodromy condition is automatically satisfied in this setting. \\
  
  Therefore, in the case where $S$ is the spectrum of a discrete valuation ring endowed with its divisorial log structure, the condition of bounded monodromy is automatically satisfied, which means that log line bundles consist of all the $\mathbb{G}_{m,log,S}$-torsors. In particular, $\mathrm{sLogPic_{X/S}}$ coincides with the log Picard functor we recalled in subsection \ref{previous work}.
 
 \end{rema}


\begin{prop}\label{finitor}
Let $C$ be a smooth projective semistable and geometrically connected curve endowed with a $K$-point. Let $\mathcal{C}$ be an $R$-regular model of $C$ with normal crossing special fiber and endowed with the divisorial log structure. Let $\mathcal{G}$ be a finite flat $R$-group scheme. Then any $R$-morphism $\mathcal{G} \to \mathrm{Pic}^{log}_{\mathcal{C}/R}$ factors through $\mathrm{sLogPic}^0_{\mathcal{C}/R}$.
 
\end{prop}

\begin{proof}
 According to the previous remark, $\mathrm{sLogPic}^0_{\mathcal{C}/R}$ is the subsheaf of $\Pic^{log}_{\mathcal{C}/R}$ of (total) degree zero log line bundles. On the other hand, since $\mathcal{G}$ is of torsion, the morphism $\mathcal{G} \to \mathrm{Pic}^{log}_{\mathcal{C}/R}$ factors through the torsion of $\Pic^{log}_{\mathcal{C}/R}$. Now, given the (total) degree map $\mathrm{Pic}^{log}_{\mathcal{C}/R} \xrightarrow{\mathrm{deg}} \mathbb{Z}$ and the fact that $\Z$ has no torsion, we deduce that torsion log line bundles have (total) degree zero. In addition, $\mathcal{G} \to \Spec(R)$ being strict, we conclude that $\mathcal{G} \to \mathrm{Pic}^{log}_{\mathcal{C}/R}$ factors through $\mathrm{sLogPic}^0_{\mathcal{C}/R}$.
\end{proof}

\begin{cor}\label{neronlogg}

Let $C$ be a smooth projective semistable and geometrically connected curve endowed with a $K$-point. Let $\mathcal{C}$ be an $R$-regular model of $C$ with normal crossing special fiber and endowed with the divisorial log structure (cf. example \ref{log sch}(3)). Let $G$ be a finite commutative $K$-group scheme and $\mathcal{G}$ a finite flat $R$-model of $G$. Then a pointed fppf $G$-torsor extends into a pointed $\mathcal{G}$-log torsor over $\mathcal{C}$ if and only if the $K$-morphism $G^D \to J$ associated to the generic torsor (cf. (\ref{Jac})) extends into an  $R$-morphism $\mathcal{G}^D \to \mathcal{J}$. 
\end{cor}
\begin{proof}
This follows from Theorem \ref{previ}, Theorem \ref{ner} and Proposition \ref{finitor}.
\end{proof}

\subsection{On the existence of a finite flat model of the group scheme}\label{finite flat tor}
We assume in this section that $k$ is algebraically closed. Let $C$ be a semistable smooth projective and geometrically connected $K$-curve with Jacobian variety $J$. Assume that $G$ is a finite commutative subgroup scheme of $J$. In particular, the  morphism $G \hookrightarrow J$ factors through $J[r]$, where $r$ is the order of $G$. If $\mathcal{J}$ is the Néron model of $J$, we let $\overline{G}$ be the schematic closure of $G$ inside $\mathcal{J}[r]$. Since $C$ is semistable, $\mathcal{J}[r]$ is flat and quasi-finite (cf. \cite[\S 7.3, Lemma~2]{BLR}). \\

In a previous paper (cf. \cite[\S 3]{Sara}), we found a necessary and sufficient condition for $\mathcal{J}[r]$ to be finite and flat\footnote{the conditions are: $C$ is semistable, together with a combinatorial condition on the dual graph of the special fiber.}, hence for $G$ to admit a finite flat $R$-model, namely $\overline{G}$.\\ 

\textbf{Question:} If $\mathcal{J}[r]$ is not assumed to be finite flat anymore, can we still find a necessary and sufficient condition for the schematic closure $\overline{G}$ of $G$ to be finite and flat?

For the rest of the section, we assume that $R$ is Henselian. Since $\mathcal{J}[r]$ is quasi-finite and $R$ Henselian, according to \cite[Lemma 1.1]{Mazur}, we have an exact sequence:
\begin{equation}\label{J[r]}
    0 \to \mathcal{FJ}[r] \to \mathcal{J}[r] \to \mathcal{EJ}[r] \to 0
\end{equation}

where $\mathcal{FJ}[r]$ is a finite flat group scheme over $R$ and $\mathcal{EJ}[r]$ is an étale group scheme over $R$, with trivial special fiber. In particular, it follows from \cite[\S IX, Lemma 2.2.3]{SGA7} that $\mathcal{FJ}[r]$ is the largest finite subgroup scheme in $\mathcal{J}[r]$.\\

 The schematic closure $\overline{G}$ of $G$ in $\mathcal{J}[r]$ is flat and quasi-finite. Hence we have as previously an exact sequence: 
 \begin{equation*}
 0 \to \mathcal{F} \to \overline{G} \to \mathcal{E} \to 0 
 \end{equation*}
where $\mathcal{F}$ is a finite flat group scheme over $R$ and $\mathcal{E}$ is an étale group scheme over $R$, with trivial special fiber. We would like to find a necessary and sufficient condition for $\overline{G}$ to be finite. \\

We denote by $\mathcal{FJ}[r]_K$ the generic fiber of $\mathcal{FJ}[r]$.

\begin{lem}\label{fini}
    $\overline{G}$ is finite if and only if  $G \to J[r]$ factors through $\mathcal{FJ}[r]_K$.
\end{lem}

\begin{proof}
    If $\overline{G}$ is finite, since $\mathcal{FJ}[r]$ is the largest finite subgroup scheme inside $\mathcal{J}[r]$, then $\overline{G} \to \mathcal{J}[r]$ factors through $\mathcal{FJ}[r]$, hence $G \to J[r]$ factors through $\mathcal{FJ}[r]_K$.
    
On the other hand, if ${G} \to {J}[r]$ factors through $\mathcal{FJ}[r]_K$, since $\mathcal{FJ}[r]$ is closed inside $\mathcal{J}[r]$ (it is a kernel), $\overline{G}$ is the schematic closure of $G$ in $\mathcal{FJ}[r]$, hence it is finite (closed immersions are finite and the composition of two finite morphisms is finite).
\end{proof}

\begin{cor}\label{Finitesemi}
    Let $C$ be a semistable smooth projective and geometrically connected curve with a $K$-point, $\mathcal{C}$ an $R$-regular model of $C$ with normal crossing special fiber and endowed with the divisorial log structure, $J$ the Jacobian of $C$ and $\mathcal{J}$ its Néron model. Let $G$ be a finite subgroup scheme of $J$. Then, the corresponding fppf pointed $G^D$-torsor $Y \to C$ (cf. (\ref{Jac})) extends into a log torsor over $\mathcal{C}$ under a finite flat group scheme if and only if $G$ is a subgroup of $\mathcal{FJ}[r]_K$, with $r$ the order of $G$. 
\end{cor}

\begin{proof}
If $G \hookrightarrow J[r]$ factors through $\mathcal{FJ}[r]_K$, then the schematic closure $\overline{G}$ of $G$ in $\mathcal{J}[r]$ is finite and flat by Lemma \ref{fini}, and it follows from Corollary \ref{neronlogg} that the torsor extends into a logarithmic $\overline{G}^D$-torsor over $\mathcal{C}$. On the other hand, if there exists a finite flat model $\mathcal{G}$ of $G$ such that the torsor extends into a log $\mathcal{G}^D$-torsor, then it follows from Corollary \ref{neronlogg} again that the $K$-morphism $G \to J[r]$ extends into an $R$-morphism $\mathcal{G} \to \mathcal{J}[r]$. Since $\mathcal{J}$ is separated, it follows that $\mathcal{G}$ is necessarily the schematic closure of $G$ in $\mathcal{J}[r]$, hence, by Lemma \ref{fini}, it implies that $G \hookrightarrow J[r]$ factors through $\mathcal{FJ}[r]_K$.

\end{proof}

\begin{Counter-example}\label{CE}
Let $A$ be a local noetherian and complete ring, with fraction field $K$, and residue field of characteristics $p$. We find in \cite[\S 5]{Oort} the following bijection 

$$\{\mathrm{isomorphism~classes~of~}A\mathrm{-group~schemes~of~order~}p \} \simeq \{(a,b) \in A^2|ab=p\}/\sim $$

where  $(a,b)\sim (c,d)$ if and only if $\exists u \in A^{\times}$ such that $c=u^{p-1}a$ and $d=u^{1-p}b$. Considering the restriction morphism 
$$\{(a,b) \in A^2|ab=p\}/\sim \xrightarrow{\varphi} \{(a,b) \in K^2|ab=p\}/\sim, $$
and taking for example $A=\mathbb{Z}_p$, it is easy to see that $\varphi$ is not surjective in general, which means that there are $\mathbb{Q}_p$-group schemes that doesn't extend into a finite flat $\mathbb{Z}_p$-group schemes.
\end{Counter-example}

\textbf{Question: } More generally, it is natural to ask what happens if don't we assume that the structural group of the torsor admits a finite flat $R$-model. We investigate this question in the next section.

\section{Part II: Extension of torsors under a quasi-finite flat group scheme}\label{quasi-finite}

The following result by Antei says that there exists some regular model of the curve where the torsor extends into an fppf torsor:
\begin{thm}\cite[Theorem 3.7]{Antei}\label{Ant}
Let $\mathcal{X} \to \Spec(R)$ be a faithfully flat morphism of finite type, with $\mathcal{X}$ a regular and integral scheme of absolute dimension $2$ endowed with an $R$-section. Let $G$ be a finite $K$-group scheme and $f:Y\to \mathcal{X}_K$ an fppf pointed $G$-torsor. Then there exists an integral scheme $\mathcal{X}_0$, faithfully flat and of finite type over $R$, a model map $\lambda: \mathcal{X}_0 \to \mathcal{X}$ and an fppf $\mathcal{G}$-torsor $\mathcal{Y} \to \mathcal{X}_0$ extending the given $G$-torsor $Y$ for some quasi-finite and flat $R$-group scheme $\mathcal{G}$.  Moreover, $\mathcal{X}_0$ can be obtained by $\mathcal{X}$ after a finite number of Néron blow-ups.
\end{thm}

On the other hand, it is shown in \cite{Pedro} that an fppf torsor under a quasi-finite flat group scheme reduces into a torsor under a finite flat group scheme: 

\begin{thm}\cite[Theorem 12.1]{Pedro}\label{Pedro}
  Let $R$ be a discrete valuation ring which is assumed to be Henselian Japanese, such that its residue field $k$ is perfect. Let $\mathcal{X}$ be a normal, irreducible, projective and flat $R$-scheme with geometrically reduced fibres and with an $R$-section. Let $\mathcal{G}$ be a quasi-finite flat $R$-group scheme and $\mathcal{Y} \to \mathcal{X}$ a pointed fppf $\mathcal{G}$-torsor. Then, there exists a finite flat $R$-group scheme $\mathcal{H}$, a morphism $\mathcal{H} \to \mathcal{G}$ and a pointed fppf $\mathcal{H}$-torsor $\mathcal{Y}_0 \to \mathcal{X}$ such that $\mathcal{Y}_0 \times^{\mathcal{H}} \mathcal{G} \simeq \mathcal{Y}$ is pointed fppf $\mathcal{G}$-torsors.
\end{thm}
We deduce from it the following:

\begin{cor}\label{Pedrobis}
   With the same notations and assumptions as in Theorem \ref{Pedro}, if $\mathcal{Y} \to \mathcal{X}$ is a pointed fppf $\mathcal{G}$-torsor and $\mathcal{F}$ the largest finite subgroup scheme inside $\mathcal{G}$, there exists a pointed fppf $\mathcal{F}$-torsor $\mathcal{Y}_0 \to \mathcal{X}$ such that $\mathcal{Y}_0 \times^{\mathcal{F}} \mathcal{G} \simeq \mathcal{Y}$ as pointed fppf $\mathcal{G}$-torsors.
       
    \end{cor}

\begin{proof}
Let $\mathcal{H}$ be the finite group scheme in Theorem \ref{Pedro}. Since $\mathcal{F}$ is the largest finite subgroup of $\mathcal{G}$, $\mathcal{H} \to \mathcal{G}$ factors through $\mathcal{F}$. Therefore, the surjective map
    \begin{align*}
        H^1_{fppf}(\mathcal{X},\mathcal{H}) & \to H^1_{fppf}(\mathcal{X},\mathcal{G})\\
        \mathcal{T} &\mapsto \mathcal{T} \times^{\mathcal{H}} \mathcal{G}
    \end{align*}
factors through $H^1_{fppf}(\mathcal{X},\mathcal{F})$ in an obvious way.  
\end{proof}

\begin{thm}\label{main thm} Let $C$ be a smooth projective and geometrically connected $K$-curve with a $K$-point. Let $\mathcal{C}$ be a regular model of $C$ and $G$ a finite commutative $K$-group scheme. Let $Y \to C$ be an fppf pointed $G$-torsor. Then, there exists a quasi-finite flat group scheme $\mathcal{G}$ over $R$ with generic fiber $G$, and an fppf pointed $\mathcal{G}$-torsor over $\mathcal{C}$ that extends the $G$-torsor $Y \to C$.
\end{thm}

\begin{proof}
    By Theorem \ref{Ant}, there exists a quasi-finite flat $R$-group scheme $\mathcal{G}$ that extends $G$, together with a regular model $\mathcal{C}_0$ of $C$ such that the fppf pointed $G$-torsor $Y \to C$ extends into an fppf $\mathcal{G}$-torsor $\mathcal{Y} \to \mathcal{C}_0$. By Corollary \ref{Pedrobis}, if $\mathcal{F}$ is the largest finite subgroup scheme of $\mathcal{G}$, there exists a pointed fppf $\mathcal{F}$-torsor $\mathcal{Y}_0 \to \mathcal{C}_0$ such that $\mathcal{Y}_0 \times^{\mathcal{F}} \mathcal{G}\simeq \mathcal{Y}$. Hence, the pointed fppf $\mathcal{F}_K$-torsor $\mathcal{Y}_{0,K} \to C$ extends into the pointed fppf $\mathcal{F}$-torsor $\mathcal{Y}_0 \to \mathcal{C}_0$. If $J$ denotes the Jacobian of $C$ and $\mathcal{J}^0$ the identity component of its Néron model, this is equivalent by Proposition \ref{PropN} to the fact that the associated $K$-morphism $\mathcal{F}_K^D \to J$ (cf.(\ref{Jac})) extends into an $R$-morphism $\mathcal{F}^D \to \mathcal{J}^0$. But this  implies by Proposition \ref{PropN} again that the fppf $\mathcal{F}_K$-torsor  $\mathcal{Y}_{0,K} \to C$ extends into an fppf $\mathcal{F}$-torsor $\mathcal{Y}' \to \mathcal{C}$. Hence, the $G$-torsor $Y \to C$ extends into the fppf $\mathcal{G}$-torsor $\mathcal{Y}' \times^{\mathcal{F}} \mathcal{G} \to \mathcal{C}$. 
\end{proof}

\begin{Backmatter}
\paragraph{Acknowledgment}
 The author would like to thank her thesis advisors Jean Gillibert and Dajano Tossici for their constant support, encouragements as well as the mathematical discussions with them. She would also like to thank Thibault Poiret with whom it was very fruitful to discuss and who notably helped her to understand some points of \cite{Holmes} that he co-wrote.

\def\refname{References}

\end{Backmatter}

\affauthor{Sara Mehidi}
\affiliation{Institut de Mathématiques de Bordeaux
351, cours de la Libération - F 33 405 TALENCE.
Bureau 315, IMB. France \protect\email{sarah.mehidi@math.u-Bordeaux.fr}}

\end{document}